\documentclass[reqno,10pt,centertags]{amsart} 
\usepackage{amsmath,amsthm,amscd,amssymb,latexsym,esint,upref,stmaryrd,
enumerate,color,verbatim,yfonts}
\usepackage{hyperref} 
\newcommand*{\mailto}[1]{\href{mailto:#1}{\nolinkurl{#1}}}
\newcommand{\arxiv}[1]{\href{http://arxiv.org/abs/#1}{arXiv:#1}}

\newcommand{\bbC}{{\mathbb{C}}}

\newcommand{\bbN}{{\mathbb{N}}}

\newcommand{\bbR}{{\mathbb{R}}}

\newcommand{\bbZ}{{\mathbb{Z}}}

\newcommand{\cH}{{\mathcal H}}

\DeclareMathOperator{\supp}{supp}

\DeclareMathOperator{\dom}{dom}

\renewcommand{\Re}{\text{\rm Re}}

\renewcommand{\ln}{\text{\rm ln}}

\newcommand{\no}{\notag}
\newcommand{\lb}{\label}
\newcommand{\f}{\frac}

\newcommand{\ol}{\overline}

\newcommand{\wti}{\widetilde}
\newcommand{\Oh}{O}
\newcommand{\oh}{o}
\newcommand{\hatt}{\widehat} 
\newcommand{\dott}{\,\cdot\,}

\renewcommand{\dot}{\overset{\textbf{\Large.}}}
\renewcommand{\ddot}{\overset{\textbf{\Large..}}}

\newcommand{\bi}{\bibitem}

\makeatletter
\def\theequation{\@arabic\c@equation}

\allowdisplaybreaks 
\numberwithin{equation}{section}

\newtheorem{theorem}{Theorem}[section]

\newtheorem{lemma}[theorem]{Lemma}
\newtheorem{corollary}[theorem]{Corollary}

\theoremstyle{remark}
\newtheorem{remark}[theorem]{Remark}

\begin{document}

\title[Bessel-type operators and a refinement of Hardy's inequality]{Bessel-Type Operators and a Refinement of Hardy's Inequality}

\author[F.\ Gesztesy]{Fritz Gesztesy}
\address{Department of Mathematics, 
Baylor University, Sid Richardson Bldg., 1410 S.~4th Street, Waco, TX 76706, USA}
\email{\mailto{Fritz\_Gesztesy@baylor.edu}}
\urladdr{\url{http://www.baylor.edu/math/index.php?id=935340}}

\author[M.\ M.\ H.\ Pang]{Michael M.\ H.\ Pang}
\address{Department of Mathematics,
University of Missouri, Columbia, MO 65211, USA}
\email{\mailto{pangm@missouri.edu}}
\urladdr{\url{https://www.math.missouri.edu/people/pang}}

\author[J.\ Stanfill]{Jonathan Stanfill}
\address{Department of Mathematics, 
Baylor University, Sid Richardson Bldg., 1410 S.~4th Street, Waco, TX 76706, USA}
\email{\mailto{Jonathan\_Stanfill@baylor.edu}}
\urladdr{\url{http://sites.baylor.edu/jonathan-stanfill/}}

\dedicatory{Dedicated with great pleasure to Lance Littlejohn on the occasion of his 70th birthday.} 
\date{\today}
\thanks{Originally appeared in {\it From Operator Theory to Orthogonal Polynomials, Combinatorics, and Number Theory.~A Volume in Honor of Lance Littlejohn's 70th Birthday}, F.\ Gesztesy and A.\ Martinez-Finkelshtein (eds.), Operator Theory: Advances and Applications, Vol.~285, Birkh\"auser, Springer, Cham, 2021, pp.~143--172. This version contains some updates.}
\@namedef{subjclassname@2020}{\textup{2020} Mathematics Subject Classification}

\begin{abstract}
The principal aim of this paper is to employ Bessel-type operators in proving the inequality
\begin{align*}
\int_0^\pi dx \, |f'(x)|^2 \geq \dfrac{1}{4}\int_0^\pi dx \, \dfrac{|f(x)|^2}{\sin^2 (x)}+\dfrac{1}{4}\int_0^\pi dx \, |f(x)|^2,\quad f\in H_0^1 ((0,\pi)),
\end{align*}
where both constants $1/4$ appearing in the above inequality are optimal. In addition, this inequality is strict in the sense that equality holds if and only if $f \equiv 0$. This inequality is derived with the help of the exactly solvable, strongly singular, Dirichlet-type Schr\"{o}dinger operator associated with the differential expression
\begin{align*}
\tau_s=-\dfrac{d^2}{dx^2}+\dfrac{s^2-(1/4)}{\sin^2 (x)}, \quad s \in [0,\infty), \; x \in (0,\pi).
\end{align*}

The new inequality represents a refinement of Hardy's classical inequality
\begin{align*}  
\int_0^\pi dx \, |f'(x)|^2 \geq \dfrac{1}{4}\int_0^\pi dx \, \dfrac{|f(x)|^2}{x^2}, \quad f\in H_0^1 ((0,\pi)), 
\end{align*}
it also improves upon one of its well-known extensions in the form 
\begin{align*}
\int_0^\pi dx \, |f'(x)|^2 \geq \dfrac{1}{4}\int_0^\pi dx \, \dfrac{|f(x)|^2}{d_{(0,\pi)}(x)^2}, 
\quad f\in H_0^1 ((0,\pi)),
\end{align*}
where $d_{(0,\pi)}(x)$ represents the distance from $x \in (0,\pi)$ to the boundary $\{0,\pi\}$ of $(0,\pi)$.
\end{abstract}

\maketitle

{\scriptsize{\tableofcontents}}
\normalsize

\section{Introduction} \lb{s1}

{\it Happy Birthday, Lance! We hope this modest contribution to Hardy-type inequalities will cause some joy.}

In a nutshell, the aim of this note is to employ a Bessel-type operator in deriving the Hardy-type inequality (see 
the footnote on p.~7),  
\begin{align}
\int_0^\pi dx \, |f'(x)|^2 \geq \dfrac{1}{4}\int_0^\pi dx \, \dfrac{|f(x)|^2}{\sin^2 (x)} 
+\dfrac{1}{4}\int_0^\pi dx \, |f(x)|^2,\quad f\in H_0^1 ((0,\pi)).     \lb{1.1} 
\end{align}

As is readily verified, \eqref{1.1} indeed represents an improvement over the classical Hardy inequality
\begin{align} \lb{1.2}
\int_0^\pi dx \, |f'(x)|^2 \geq \dfrac{1}{4}\int_0^\pi dx \, \dfrac{|f(x)|^2}{x^2}, \quad f\in H_0^1 ((0,\pi)), 
\end{align}
while also improving upon one of its well-known extensions in the form
\begin{align} \lb{1.3}
\int_0^\pi dx \, |f'(x)|^2 \geq \dfrac{1}{4}\int_0^\pi dx \, \dfrac{|f(x)|^2}{d_{(0,\pi)}(x)^2},\quad f\in H_0^1 ((0,\pi)).
\end{align}
Here $d_{(0,\pi)}(x)$ represents the distance from $x \in (0,\pi)$ to the boundary $\{0,\pi\}$ of the interval $(0,\pi)$, that is, 
\begin{align} \lb{1.4}
d_{(0,\pi)}(x)=\begin{cases} x, & x\in(0,\pi/2],\\
\pi-x, & x\in[\pi/2,\pi). 
\end{cases}
\end{align}
We emphasize that all constants $1/4$ in \eqref{1.1}--\eqref{1.3} are optimal and all inequalities are strict in the sense that equality holds in them if and only if $f \equiv 0$.

Our refinement \eqref{1.1} (and the optimality of both constants $1/4$ in \eqref{1.1}) rests on the exact solvability of the one-dimensional Schr\"odinger equation with potential $q_s$, $s \in [0,\infty)$, given by 
\begin{equation}
q_s(x) = \f{s^2 - (1/4)}{\sin^2(x)}, \quad x \in (0,\pi),    \lb{1.5} 
\end{equation}
as illustrated by Rosen and Morse \cite{RM32} in 1932, P\"oschl and Teller \cite{PT33} in 1933, and Lotmar \cite{Lo35} in 1935. These authors are either concerned with the following extension of \eqref{1.5} 
\begin{equation} 
\f{c_1}{\sin^2(x)} + \f{c_2}{\cos^2(x)}, \quad x \in (0,\pi/2),    \lb{1.6}
\end{equation}
or its hyperbolic analog of the form
\begin{equation} 
\f{c_1}{\sinh^2(x)} + \f{c_2}{\cosh^2(x)}, \quad x \in \bbR \; (\text{or } \, x \in (0,\infty)).    \lb{1.7} 
\end{equation}
The upshot of these investigations for the purpose at hand was the realization that such problems are exactly solvable in terms of the hypergeometric function $F(a,b;c;\dott)$ (frequently denoted by $\mathstrut_2F_1(a,b;c;\dott)$). These types of problems are further discussed by Infeld and Hull \cite{IH51} and summarized in  \cite[Sect.~38, 39, 93]{Fl99}, and more recently in  \cite{DW11}. A discussion of the underlying singular periodic problem \eqref{1.5} on $\bbR$, including the associated Floquet (Bloch) theory, was presented by Scarf \cite{Sc58}. These investigations exclusively focus on aspects of ordinary differential equations as opposed to operator theory even though Dirichlet problems associated with singular endpoints were formally discussed (in this context see also \cite{Sc58a}). The operator theoretic approach to \eqref{1.5} and \eqref{1.6} over a finite interval bounded by singularities, and a variety of associated self-adjoint boundary conditions including coupled boundary conditions leading to energy bands 
(Floquet--Bloch theory) in the periodic problem on $\bbR$, on the basis of generalized boundary values due to Rellich \cite{Re43} (see also \cite{BG85}), was first discussed in \cite{GK85} and \cite{GMS85}. Finally, we briefly mention that the case of $n$-soliton potentials  
\begin{equation} 
q_{(1/2)+n}(x) = n(n+1)/\cosh^2(x), \quad n \in \bbN, \; x \in \bbR, 
\end{equation} 
has received special attention as it represents a solution of infinitely many equations in the stationary Korteweg--de Vries (KdV) hierarchy (starting from level $n$ upward). 

Introducing the differential expression 
\begin{equation}
\tau_s = - \f{d^2}{dx^2} + q_s(x) = - \f{d^2}{dx^2} + \f{s^2 - (1/4)}{\sin^2(x)}, \quad x \in (0,\pi), 
\end{equation}
the exact solvability of the differential equation $\tau_s y = z y$, $z \in \bbC$, or a comparison with the well-known 
Bessel operator case $- (d^2/dx^2) + \big[s^2 - (1/4)\big] x^{-2}$ near $x = 0$ and 
$- (d^2/dx^2) + \big[s^2 - (1/4)\big] (x - \pi)^{-2}$ near $x = \pi$ then yields the nonoscillatory property of $\tau_s$ 
if and only if $s \in [0,\infty)$. Very roughly speaking, nonnegativity of the Friedrichs extension associated with the differential expression $\tau_0 - (1/4)$, implying nonnegativity of the underlying quadratic form defined on $H_0^1 ((0,\pi))$, implies the refinement \eqref{1.1} of Hardy's inequality. 

In Section \ref{s2} we briefly discuss (principal and nonprincipal) solutions of the exactly solvable Schr\"odinger equation $\tau_s y = 0$ (solutions of the general equation $\tau_s y = z y$, $z \in \bbC$, are discussed in Appendix \ref{sA}), and introduce minimal $T_{s,min}$ and maximal $T_{s,max} = T_{s,min}^*$ operators corresponding to $\tau_s$ as well as the Friedrichs extension $T_{s,F}$ of $T_{s,min}$ and the boundary values associated with $T_{s,max}$, following recent treatments in \cite{GLN20}, \cite{GLPS21}. Section \ref{s3} contains the bulk of this paper and is devoted to a derivation of inequality \eqref{1.1}. We also indicate how two related results by Avkhadiev and Wirths \cite{AW07}, \cite{AW11}, involving Drichlet boundary conditions on both ends and a mixture of Dirichlet and Neumann boundary conditions naturally fits into the framework discussed in this paper. In Appendix \ref{sA} we study solutions of $\tau_s y = z y$, $z \in \bbC$, in more detail and also derive the singular Weyl--Titchmarsh--Kodaira $m$-function associated with $T_{s,F}$. Finally, Appendix \ref{sB} collects some facts on Hardy-type inequalities.

\section{An Exactly Solvable, Strongly Singular, \\ Periodic Schr\"odinger Operator}\lb{s2}

In this section we examine a slight variation of the example found in Section 4 of \cite{GK85} by implementing the methods found in \cite{GLN20}.

Let $a=0$, $b=\pi,$
\begin{equation}
p(x)=r(x)=1,\quad q_s(x)=\f{s^2-(1/4)}{\sin^2 (x)},\quad s \in [0,\infty), \; x\in(0,\pi).    \lb{2.1}
\end{equation}
We now study the Sturm--Liouville operators associated with the corresponding differential expression given by
\begin{align}
\tau_s = - \dfrac{d^2}{dx^2} + q_s(x) 
= - \dfrac{d^2}{dx^2}+\dfrac{s^2-(1/4)}{\sin^2 (x)}, \quad s \in [0,\infty), \; x\in(0,\pi),    \lb{2.2} 
\end{align}
which is in the limit circle case at the endpoints $x=0,\pi$ for $s\in[0,1)$, and limit point at both endpoints for 
$s \in [1,\infty)$. The maximal and preminimal operators, $T_{s,max}$ and $\dot T_{s,min}$, associated to $\tau_{s}$ 
in $L^2((0,\pi); dx)$ are then given by  
\begin{align}
&T_{s,max} f = \tau_{s} f, \quad s \in [0,\infty),     \no \\
& f \in \dom(T_{s,max})=\big\{g\in L^2((0,\pi); dx) \, \big| \,  g,g'\in AC_{loc}((0,\pi));     \lb{2.3} \\
& \hspace*{6.35cm}  \tau_{s} g\in L^2((0,\pi); dx)\big\},     \no 
\end{align}
and
\begin{align}
&\dot T_{s,min} f = \tau_{s} f, \quad s \in [0,\infty),   \no
\\
& f \in \dom\big(\dot T_{s,min}\big)=\big\{g\in L^2((0,\pi); dx)  \, \big| \,  g,g'\in AC_{loc}((0,\pi));   
\lb{2.4} \\ 
&\hspace*{1.9cm} \supp \, (g)\subset(0,\pi) \text{ is compact; } 
\tau_{s} g\in L^2((0,\pi); dx)\big\}.   \no 
\end{align}
Since 
\begin{equation}
q_s \in L^2_{loc}((0,\pi); dx), \quad s \in [0,\infty),     \lb{2.5} 
\end{equation}
one can replace $\dot T_{s,min}$ by
\begin{equation}
\ddot T_{s,min} = \tau_{s}\big|_{C_0^{\infty}((0,\pi))}, \quad s \in [0,\infty).  \lb{2.6} 
\end{equation}

For $s\in[0,1)$, we introduce principal and nonprincipal solutions $u_{0,s}(0,\dott)$ and $\hatt  u_{0,s}(0,\dott)$ of $\tau_s u=0$ at $x=0$ by
\begin{align}
u_{0,s}(0,x)&=[\sin (x)]^{(1+2s)/2}F\big((1/4)+(s/2),(1/4)+(s/2);1+s;\sin^2 (x)\big),   \no \\
& \hspace*{8.35cm}  s\in[0,1),\no  \\[1mm]
\hatt  u_{0,s}(0,x)&=\begin{cases}
(2s)^{-1}[\sin (x)]^{(1-2s)/2}    \\
\quad \times F\big((1/4)-(s/2),(1/4)-(s/2);1-s;\sin^2 (x)\big),  \quad s\in(0,1),     \\[1mm]
[\sin (x)]^{1/2}F\big(1/4,1/4;1;\sin^2 (x)\big)       \\
\quad \times\displaystyle\int_x^{c} dx' \, [\sin (x')]^{-1}\big[F\big(1/4,1/4;1;\sin^2 (x') \big)\big]^{-2}, \quad s=0,
\end{cases}    \lb{2.7} 
\end{align}
and principal and nonprincipal solutions $u_{\pi,s}(0,\dott)$ and $\hatt  u_{\pi,s}(0,\dott)$ of $\tau_s u=0$ at $x=\pi$ by
\begin{align}
u_{\pi,s}(0,x)&=[\sin (x)]^{(1+2s)/2}F\big((1/4)+(s/2),(1/4)+(s/2);1+s;\sin^2 (x)\big),   \no \\ 
& \hspace*{8.3cm} s\in[0,1),\no \\[1mm]
\hatt  u_{\pi,s}(0,x)&=\begin{cases}
-(2s)^{-1} [\sin (x)]^{(1-2s)/2}    \\ 
\quad \times F\big((1/4)-(s/2),(1/4)-(s/2);1-s;\sin^2 (x)\big), \quad s\in(0,1),\\[1mm]
- [\sin (x)]^{1/2}F\big(1/4,1/4;1;\sin^2 (x)\big)\\
\quad\times\displaystyle\int_c^x dx' \, [\sin (x')]^{-1}\big[F\big(1/4,1/4;1;\sin^2 (x') \big)\big]^{-2}, \quad s=0.
\end{cases}       \lb{2.8} 
\end{align}
Here $F(\dott,\dott,\dott;\dott)$ $($frequently written as $\mathstrut_2F_1(\dott,\dott;\dott;\dott)$$)$ denotes the hypergeometric function $($see, e.g., \cite[Ch.~15]{AS72}$)$. 

\begin{remark} \lb{r2.1} 
We note that the case $c=1$ in $F(a,b;c;\xi)$, corresponding to the case $s=0$ in \eqref{2.7}, \eqref{2.8}, is a  special one in the sense that linearly independent solutions of the hypergeometric differential equation are then of the form (see, e.g., \cite[Nos.~15.5.16, 15.5.17]{AS72})
\begin{align}
& y_1(\xi) = F(a,b;1;\xi),     \no \\
& y_2(\xi) =  F(a,b;1;\xi) \ln(\xi)    \lb{2.8a} \\
& \quad + \sum_{n\in \bbN} \f{(a)_n (b)_n}{(n!)^2} [\psi(a+n) - \psi(a) + \psi(b+n) - \psi(b) 
+ 2 \psi(1) - 2 \psi(n+1)] \xi^n.    \no
\end{align}
Here $(d)_n$, $n \in \bbN_0$, represents Pochhammer's symbol (see, \eqref{A.5}), and $\psi(\dott)$ denotes the Digamma function. Since we wanted to ensure \eqref{2.11} and our principal aim in connection with the boundary values \eqref{2.13}--\eqref{2.16} was the derivation of the asymptotic relations \eqref{2.9} and \eqref{2.10}, our choice of $\hatt u_{0,0}$ and $\hatt u_{\pi,0}$ in \eqref{2.7} and \eqref{2.8} is to be preferred over the use of the pair of functions in \eqref{2.8a}. For more details in this connection see Appendix \ref{sA}.  
\hfill $\diamond$
\end{remark}

Since
\begin{align}
u_{0,s}(0,x)&\underset{x\downarrow0}{=}x^{(1+2s)/2}\big\{1+\big[\big(4s^2-1\big)/\big(48+48s\big)\big]x^2+\Oh\big(x^4\big)\big\},    \quad s\in[0,1),   \no \\[1mm]
\hatt  u_{0,s}(0,x)&\underset{x\downarrow0}{=}\begin{cases}
(2s)^{-1}x^{(1-2s)/2}\big\{1+\big[\big(4s^2-1\big)/\big(48-48s\big)\big]x^2+\Oh\big(x^4\big)\big\}, \\
\hspace*{7.23cm}  s\in(0,1),\\[1mm]
\ln(1/x)x^{1/2}\big\{1+\big[\big([\ln(x)]^{-1}-1\big)/48\big]x^2+\Oh\big(x^4\big)\big\}, \quad s=0,
\end{cases}      \lb{2.9} \\[1mm]
u_{\pi,s}(0,x)&\underset{x\uparrow\pi}{=}(\pi-x)^{(1+2s)/2}\big\{1+\big[\big(4s^2-1\big)/\big(48+48s\big)\big](\pi-x)^2     \no \\
& \hspace*{3.1cm} +\Oh\big((\pi-x)^4\big)\big\},   \quad  s\in[0,1),\no \\[1mm]
\hatt  u_{\pi,s}(0,x)&\underset{x\uparrow\pi}{=}\begin{cases}
-(2s)^{-1}(\pi-x)^{(1-2s)/2}\big\{1+\big[\big(4s^2-1\big)/\big(48-48s\big)\big](\pi-x)^2\\
\quad+\Oh\big((\pi-x)^4\big)\big\},  \quad  s\in(0,1),\\[1mm]
\ln(\pi-x)(\pi-x)^{1/2}\big\{1+\big[\big([\ln(\pi-x)]^{-1}-1\big)/48\big](\pi-x)^2\\
\quad +\Oh\big((\pi-x)^4\big)\big\},   \quad s=0,
\end{cases}       \lb{2.10} 
\end{align}
one deduces that
\begin{align}
W(\hatt  u_{0,s}(0,\dott),u_{0,s}(0,\dott))(0)=1=W(\hatt  u_{\pi,s}(0,\dott),u_{\pi,s}(0,\dott))(\pi),\quad s\in[0,1),     \lb{2.11} 
\end{align}
and
\begin{align}
\lim_{x\downarrow0}\dfrac{u_{0,s}(0,x)}{\hatt  u_{0,s}(0,x)}=0,\quad \lim_{x\uparrow\pi}\dfrac{u_{\pi,s}(0,x)}{\hatt  u_{\pi,s}(0,x)}=0,\quad s\in[0,1).      \lb{2.12} 
\end{align}

The generalized boundary values for $g\in\dom(T_{s,max})$ (the maximal operator associated with $\tau_s$) are then of the form
\begin{align}
\wti g(0)&=\begin{cases}
\lim_{x\downarrow0}g(x)/\big[(2s)^{-1}x^{(1-2s)/2}\big], & s\in(0,1),\\[1mm]
\lim_{x\downarrow0}g(x)/\big[x^{1/2}\ln(1/x)\big], & s=0,
\end{cases}    \lb{2.13} \\[1mm]
\wti g^{\, \prime}(0)&=\begin{cases}
\lim_{x\downarrow0}\big[g(x)-\wti g(0)(2s)^{-1}x^{(1-2s)/2}\big]\big/x^{(1+2s)/2}, & s\in(0,1),\\[1mm]
\lim_{x\downarrow0}\big[g(x)-\wti g(0)x^{1/2}\ln(1/x)\big]\big/x^{1/2}, & s=0,
\end{cases}    \lb{2.14} \\[1mm]
\wti g(\pi)&=\begin{cases}
\lim_{x\uparrow\pi}g(x)/\big[-(2s)^{-1}(\pi-x)^{(1-2s)/2}\big], & s\in(0,1),\\[1mm]
\lim_{x\uparrow\pi}g(x)/\big[(\pi-x)^{1/2}\ln(\pi-x)\big], & s=0,
\end{cases}     \lb{2.15} \\[1mm]
\wti g^{\, \prime}(\pi)&=\begin{cases}
\lim_{x\uparrow\pi}\big[g(x)+\wti g(\pi)(2s)^{-1}(\pi-x)^{(1-2s)/2}\big]\big/(\pi-x)^{(1+2s)/2}, & s\in(0,1),\\[1mm]
\lim_{x\uparrow\pi}\big[g(x)-\wti g(0)(\pi-x)^{1/2}\ln(\pi-x)\big]\big/(\pi-x)^{1/2}, & s=0.
\end{cases}    \lb{2.16} 
\end{align}
As a result, the minimal operator 
$T_{s,min}$ associated to $\tau_{s}$, that is, 
\begin{equation}
T_{s,min} = \ol{\dot T_{s,min}} = \ol{\ddot T_{s,min}}, \quad s \in [0,\infty),    \lb{2.17} 
\end{equation} 
is thus given by 
\begin{align}
&T_{s,min} f = \tau_{s} f,     \no \\
& f \in \dom(T_{s,min})= \big\{g\in L^2((0,\pi); dx) \, \big| \, g,g'\in AC_{loc}((0,\pi));    \lb{2.18} \\ 
& \hspace*{1.1cm}  \wti g(0) = {\wti g}^{\, \prime}(0) = \wti g(\pi) = {\wti g}^{\, \prime}(\pi) = 0; \, 
\tau_{s} g\in L^2((0,\pi); dx)\big\},  \quad s \in [0,1),     \no 
\end{align}
and satisfies $T_{s,min}^* = T_{s,max}$, $T_{s,max}^* = T_{s,min}$, $s \in [0,\infty)$. 
Due to the limit point property of $\tau_s$ at $x=0$ and $x=\pi$ if and only if $s \in [1,\infty)$, one concludes that 
\begin{equation}
T_{s,min} = T_{s,max} \, \text{ if and only if } \, s \in [1,\infty).   \lb{2.19}
\end{equation}

The Friedrichs extension $T_{s,F}$ of $T_{s,min}$, $s \in [0,1)$, permits a particularly simple characterization in terms of the generalized boundary conditions \eqref{2.13}--\eqref{2.16} and is then given by (cf.\ \cite{Ka78}, \cite{NZ92}, \cite{Re51}, \cite{Ro85} and the extensive literature cited in \cite{GLN20}, \cite[Ch.~13]{GZ21})  
\begin{align}
T_{s,F}f=\tau_s f,\quad f\in\dom(T_{s,F})=\big\{g\in\dom(T_{s,max}) \, \big| \,  \wti g(0)=\wti g(\pi)=0\big\}, 
\quad s \in [0,1),      \lb{2.20} 
\end{align}
moreover,
\begin{equation}
T_{s,F} = T_{s,min} = T_{s,max}, \quad s \in [1,\infty),     \lb{2.21} 
\end{equation}
is self-adjoint (resp., $\dot T_{s,min}$ and $\ddot T_{s,min}$, $s \in [1,\infty)$, are essentially self-adjoint) in $L^2((0,\pi); dx)$. In this case the Friedrichs boundary conditions in \eqref{2.20} are automatically satisfied and hence can be omitted. 

By \eqref{A.13} one has 
\begin{equation}
\inf(\sigma(T_{s,F})) = [(1/2) + s]^2, \quad s \in [0,\infty),     \lb{2.22}
\end{equation}
in particular, 
\begin{equation}
T_{s,F} \geq [(1/2) + s]^2 I_{(0,\pi)}, \quad s \in [0,\infty),     \lb{2.23} 
\end{equation}
with $I_{(0,\pi)}$ abbreviating the identity operator in $L^2((0,\pi); dx)$. 

All results on 2nd order differential operators employed in this section can be found in classical sources such as \cite[Sect.~129]{AG81}, \cite[Chs.~8, 9]{CL85}, \cite[Sects.~13.6, 13.9, 13.10]{DS88}, \cite[Ch.~III]{JR76}, \cite[Ch.~V]{Na68}, \cite{NZ92}, \cite[Ch.~6]{Pe88}, \cite[Ch.~9]{Te14}, \cite[Sect.~8.3]{We80}, \cite[Ch.~13]{We03}, \cite[Chs.~4, 6--8]{Ze05}. In addition, \cite{GLN20} and \cite[Ch.~13]{GZ21} contain very detailed lists of references in this context.

\section{A Refinement of Hardy's Inequality}\lb{s3}

 The principal purpose of this section is to derive a refinement of the classical Hardy inequality
\begin{align} \lb{3.1}
\int_0^\pi dx \, |f'(x)|^2 \geq \dfrac{1}{4}\int_0^\pi dx \, \dfrac{|f(x)|^2}{x^2}, \quad f\in H_0^1 ((0,\pi)), 
\end{align}
as well as of one of its well-known extensions in the form
\begin{align}\lb{3.2}
\int_0^\pi dx \, |f'(x)|^2 \geq \dfrac{1}{4}\int_0^\pi dx \, \dfrac{|f(x)|^2}{d_{(0,\pi)}(x)^2},\quad f\in H_0^1 ((0,\pi)),
\end{align}
where $d_{(0,\pi)}(x)$ represents the distance from $x \in (0,\pi)$ to the boundary $\{0,\pi\}$ of the interval $(0,\pi)$, that is, 
\begin{align}\lb{3.3}
d_{(0,\pi)}(x)=\begin{cases} x, & x\in(0,\pi/2],\\
\pi-x, & x\in[\pi/2,\pi). 
\end{cases}
\end{align}
The constant $1/4$ in \eqref{3.1} and \eqref{3.2} is known to be optimal and both inequalities are strict in the sense that equality holds in them if and only if $f \equiv 0$.

For background on Hardy-type inequalities we refer, for instance, to \cite[p.~3--5]{BEL15}, \cite{BM97}, \cite{CEL99}, \cite[p.~104--105]{Da95}, \cite{GLMW18, GGMR19, Ha25}, 
\cite[Sect.~7.3, p.~240--243]{HLP88}, \cite[Sect.~5]{Ku85}, \cite[Ch.~3]{KMP07}, \cite[Ch.~1]{KPS17}, \cite{La26}, \cite[Ch.~1]{OK90}, \cite{PS15}. 

The principal result of this section then can be formulated as follows\footnote{After our paper got published, we received a message from Professor Avkhadiev who kindly pointed out to us that he proved Theorem \ref{t3.1} in 2015, see \cite[Lemma~1]{Av15}.}. 

\begin{theorem} \lb{t3.1} 
Let $f\in H_0^1 ((0,\pi))$. Then,
\begin{equation} \lb{3.4}
\int_0^\pi dx \, |f'(x)|^2 \geq \dfrac{1}{4}\int_0^\pi dx \, \dfrac{|f(x)|^2}{\sin^2 (x)}+\dfrac{1}{4}\int_0^\pi dx \, |f(x)|^2,
\end{equation}
where both constants $1/4$ in \eqref{3.4} are optimal. In addition, the inequality is strict in the sense that equality holds in \eqref{3.4} if and only if $f \equiv 0$. 
\end{theorem}
\begin{proof}
By Section \ref{s2} for $s=0$ and by \cite[Sect. 4]{GK85} for $s \in (0,\infty)$, one has
\begin{align} \lb{3.5}
\bigg(-\dfrac{d^2}{dx^2}+\dfrac{s^2-(1/4)}{\sin^2 (x)} - [(1/2) + s]^2 I_{(0,\pi)}
\bigg)\bigg|_{C_0^\infty ((0,\pi))}\geq 0,\quad s \in [0,\infty).
\end{align}
Thus, setting $s=0$ in \eqref{3.5} yields 
\begin{align}\lb{3.6}
\int_0^\pi dx \, |f'(x)|^2 \geq \dfrac{1}{4}\int_0^\pi dx \, \dfrac{|f(x)|^2}{\sin^2 (x)}+\dfrac{1}{4}\int_0^\pi dx \, |f(x)|^2,\quad f\in C_0^\infty ((0,\pi)).
\end{align}

Now denote by $H_0^1 ((0,\pi))$ the standard Sobolev space on $(0,\pi)$ obtained upon completion of $C_0^\infty ((0,\pi))$ in the norm of $H^1 ((0,\pi))$. Since $C_0^\infty ((0,\pi))$ is dense in $H_0^1 ((0,\pi))$, given $f\in H_0^1 ((0,\pi))$, there exists a sequence $\{f_n\}_{n\in\bbN}\subset C_0^\infty ((0,\pi))$ such that $\lim_{n\to\infty}\| f_n-f\|^2_{H_0^1 ((0,\pi))}=0$. Hence, one can find a subsequence $\{f_{n_p}\}_{p\in\bbN}$ of $\{f_n\}_{n\in\bbN}$ such that $f_{n_p}$ converges to $f$ pointwise almost everywhere on $(0,\pi)$ as $p \to \infty$. Thus an application of Fatou's lemma (cf., e.g., \cite[Corollary 2.19]{Fo99}) yields that \eqref{3.6} extends to $f\in H_0^1 ((0,\pi))$, namely,
\begin{align}
& \dfrac{1}{4}\int_0^\pi dx \, \dfrac{|f(x)|^2}{\sin^2 (x)}+\dfrac{1}{4}\int_0^\pi dx \, |f(x)|^2 \no \\
& \quad \leq \liminf_{p\to\infty}\dfrac{1}{4}\int_0^\pi dx \, \dfrac{|f_{n_p}(x)|^2}{\sin^2 (x)}+\liminf_{p\to\infty}\dfrac{1}{4}\int_0^\pi dx \, |f_{n_p}(x)|^2 \quad \text{(by Fatou's lemma)} \no \\
& \quad \leq \liminf_{p\to\infty}\bigg\{\dfrac{1}{4}\int_0^\pi dx \, \dfrac{|f_{n_p}(x)|^2}{\sin^2 (x)}+\dfrac{1}{4}\int_0^\pi dx \, |f_{n_p}(x)|^2 \bigg\} \no \\
&\quad \leq \liminf_{p\to\infty}\int_0^\pi dx \, |f'_{n_p}(x)|^2 \quad \text{(by \eqref{3.6})} \no \\
& \quad =\lim_{p\to\infty}\int_0^\pi dx \, |f'_{n_p}(x)|^2 \no \\
& \quad =\int_0^\pi dx \, |f'(x)|^2.
\end{align}

The substitution $s\mapsto is$ in \eqref{2.7} results in solutions that have oscillatory behavior due to the 
factor $[\sin (x)]^{\pm is}$ in \eqref{A.3}, \eqref{A.4}, rendering all solutions of $\tau_s y(\lambda,\dott) = \lambda y(\lambda,\dott)$ oscillatory for each $\lambda \in \bbR$ if and only if $s^2<0$. Classical  oscillation theory results (see, e.g.,  \cite[Theorem 4.2]{GLN20}) prove that $\dot T_{s,min}$, and hence $T_{s,min}$ are bounded from below if and only if $s \in [0,\infty)$. This proves that the first constant $1/4$ on the right-hand side of \eqref{3.4} is optimal. 

Next we demonstrate that also the second constant $1/4$ on the right-hand side of \eqref{3.4} is optimal arguing by contradiction as follows: Suppose that for some 
$\varepsilon > 0$, 
\begin{align}   
\int_0^\pi dx \, |f'(x)|^2 \geq \dfrac{1}{4}\int_0^\pi dx \, \dfrac{|f(x)|^2}{\sin^2 (x)} 
+ \bigg(\f{1}{4} + \varepsilon\bigg) \int_0^\pi dx \, |f(x)|^2, \quad f \in \dom(T_{0,min}).   \lb{3.8}
\end{align}
Upon integrating by parts in the left-hand side of \eqref{3.8} this implies 
\begin{equation}
T_{0,min} \geq  \bigg(\f{1}{4} + \varepsilon\bigg) I_{(0,\pi)}, 
\end{equation}
implying 
\begin{equation}
T_{0,F} \geq  \bigg(\f{1}{4} + \varepsilon\bigg) I_{(0,\pi)} 
\end{equation}
(as $T_{0,min}$ and $T_{0,F}$ share the same lower bound by general principles), contradicting 
\eqref{2.22} for $s=0$. Hence also the 2nd constant $1/4$ on the right-hand side of \eqref{3.4} is optimal.

It remains to prove strictness of inequality \eqref{3.4} if $f \not\equiv 0$: Arguing again by contradiction, we suppose there exists $0 \neq f_0 \in H^1_0((0,\pi))$ such that 
\begin{equation} \lb{3.11}
\int_0^\pi dx \, |f_0'(x)|^2 = \dfrac{1}{4}\int_0^\pi dx \, \dfrac{|f_0(x)|^2}{\sin^2 (x)}
+\f{1}{4}\int_0^\pi dx \, |f_0(x)|^2.
\end{equation}
Since $H_0^1((0,\pi)) \subseteq \dom\big(T^{1/2}_{s,F}\big)$, $s \in [0,\infty)$ (in fact, one even has the equality $H_0^1((0,\pi)) = \dom\big(T^{1/2}_{s,F}\big)$ for all $ s \in (0,\infty)$, see, e.g., \cite{AB16}, \cite{BDG11}, 
\cite{DG21}, \cite{GP79}, \cite{Ka78}, \cite{KT11}), one concludes via \eqref{3.11} that
\begin{equation}
\big(T_{0,F}^{1/2} f_0, T_{0,F}^{1/2} f_0\big)_{L^2((0,\pi);dx)} = (1/4) \|f_0\|_{L^2((0,\pi);dx)}^2.  \lb{3.12} 
\end{equation}
Moreover, since $T_{0,F}$ is self-adjoint with purely discrete and necessarily simple spectrum, $T_{0,F}$ has the spectral representation 
\begin{equation}
T_{0,F} = \sum_{n \in \bbN_0} \lambda_n P_n, \quad \lambda_0 = 1/4 < \lambda_1 < \lambda_2 < \cdots,
\end{equation}
where $\sigma(T_{0,F}) = \{\lambda_n\}_{n \in \bbN_0}$ and $P_n$ are the one-dimensional projections onto the eigenvectors associated with the eigenvalues $\lambda_n$, $n \in \bbN_0$, explicitly listed in \eqref{A.13}, in particular, $\lambda_0 = 1/4$. Thus,
\begin{align}
\big(T_{0,F}^{1/2} f_0, T_{0,F}^{1/2} f_0\big)_{L^2((0,\pi);dx)} &= 
\sum_{n \in \bbN_0} \lambda_n (f_0, P_n f_0)_{L^2((0,\pi);dx)}     \no \\ 
& > \lambda_0 \sum_{n \in \bbN_0} (f_0, P_n f_0)_{L^2((0,\pi);dx)}    \no \\
& = \lambda_0 \|f_0\|_{L^2((0,\pi);dx)}^2 = (1/4) \|f_0\|_{L^2((0,\pi);dx)}^2
\end{align}
contradicting \eqref{3.12} unless
\begin{equation}
P_n f_0 = 0, \; n \in \bbN, \, \text{ and hence, } \, P_0 f_0 = f_0,   
\end{equation}
that is,
\begin{equation}
f_0 \in \dom(T_{0,F}) \, \text{ and } \, T_{0,F} f_0 = (1/4) f_0,    \lb{3.16} 
\end{equation}
employing $\lambda_0 = 1/4$. However, \eqref{3.16} implies that 
\begin{equation}
f_0(x) \underset{x \downarrow 0}{=} c x^{1/2} \big[1 + \Oh\big(x^2\big)\big] \, \text{ and hence, } \, 
f_0 \notin H_0^1((0,\pi)), 
\end{equation}
a contradiction. 
\end{proof}

\begin{remark} \lb{r3.2} 
$(i)$ That inequality \eqref{3.4} represents an improvement over the previously well-known cases \eqref{3.1} and \eqref{3.2} can be shown as follows: Since trivially
\begin{equation}
\sin(x) \leq x, \quad x \in [0,\pi],
\end{equation}
inequality \eqref{3.4} is obviously an improvement over the classical Hardy inequality \eqref{3.1}. On the other hand, since also 
\begin{align} \lb{3.19}
\sin (x) \leq \begin{cases}
x, & x\in[0,\pi/2],\\
\pi-x, & x\in[\pi/2,\pi],
\end{cases} 
\end{align}
that is (cf.\ \eqref{3.3}), 
\begin{equation}
\sin(x) \leq d_{(0,\pi)}(x), \quad x \in [0,\pi],  
\end{equation} 
inequality \eqref{3.4} also improves upon the refinement \eqref{3.2}. \\[1mm] 
$(ii)$ Assuming $a, b \in \bbR$, $a < b$, the elementary change of variables 
\begin{align}
\begin{split} 
& (0,\pi) \ni x \mapsto \xi(x) = [(b-a)x + a \pi]/\pi \in (a,b),   \\
& f(x) = F(\xi), 
\end{split} 
\end{align}
yields 
\begin{align}
\begin{split} 
\int_a^b d\xi \, |F'(\xi)|^2 & \geq \f{\pi^2}{4(b-a)^2} \int_a^b d \xi \, \f{|F(\xi)|^2}{\sin^2(\pi (\xi-a)/(b-a))}  \\
& \quad + \f{\pi^2}{4(b-a)^2} \int_a^b d \xi \, |F(\xi)|^2, \quad F \in H_0^1((a,b)). 
\end{split} 
\end{align} 
These scaling arguments apply to all Hardy-type inequalities considered in this paper and hence it suffices to restrict ourselves to convenient fixed intervals such as $(0,\pi)$, etc. 
\hfill$\diamond$
\end{remark}

\begin{remark} \lb{r3.2a} 
An earlier version of our preprint contained the following factorization of $\tau_s - (s+(1/2))^2$ into
\begin{equation}
\delta^+_s \delta_s^{} = \tau_s-(s+(1/2))^2 = - \f{d^2}{dx^2} + \f{s^2 - (1/4)}{\sin^2(x)}-(s+(1/2))^2, \quad s \in [0,\infty), \; x \in (0,\pi),  
\lb{B.3} 
\end{equation}
where the differential expressions $\delta_s$, $\delta_s^+$ are given by 
\begin{equation}
\delta_s = \f{d}{dx} - [s + (1/2)] \cot(x), \quad \delta^+_s = - \f{d}{dx} - [s + (1/2)] \cot(x), \quad 
s \in [0,\infty), \; x \in (0,\pi).      \lb{B.2} 
\end{equation} 
Thus, $\delta^+_s \delta_s^{}\big|_{C_0^{\infty}((0,\pi))} \geq 0$ yields $\tau_s\big|_{C_0^{\infty}((0,\pi))} \geq (s+(1/2))^2 I$ and taking $s=0$ implies inequality \eqref{3.4} for $f \in C_0^{\infty}((0,\pi))$ and hence for $f \in H_0^1((0,\pi))$ by the usual Fatou-type argument. Hence, if one is primarily interested in the refined Hardy inequality \eqref{3.4} itself, taking $s=0$ in \eqref{B.3} appears to be its quickest derivation. We're indebted to Ari Laptev for independently pointing this out to us which resulted in our reintroducing the factorization \eqref{B.3}. 

Considering
\begin{equation}
y_s(x) = [\sin(x)]^{(1 + 2s)/2}, \quad s \in [0,\infty), \; x \in (0,\pi),    \lb{B.4}
\end{equation}
one confirms that $\delta_s y_s = 0$, $s \in [0,\infty)$, and hence
\begin{equation}
(\tau_s y_s)(x) = [s + (1/2)]^2 y_s(x), \quad s \in [0,\infty), \; x \in (0,\pi).     \lb{B.5}
\end{equation}
A second linearly independent solution of $\tau_s y =  [s + (1/2)]^2 y$ is then given by
\begin{align} 
\begin{split} 
\hatt y_s (x) &= [\sin(x)]^{(1 + 2s)/2} \int_x^{\pi/2} dt \, [\sin(t)]^{-(1+ 2s)}    \lb{B.6} \\
& \hspace*{-0.7mm} 
\underset{x \downarrow 0}{=} \begin{cases} (2s)^{-1} x^{(1-2s)/2}[1 + \oh(1)], & s \in (0,\infty), \\
x^{1/2} \ln(1/x)[1 + \oh(1)], & s = 0. 
\end{cases}
\end{split} 
\end{align}
By inspection, $y_s' \in L^2((0,\pi);dx)$ if and only if $s \in (0,\infty)$, and hence there is a cancellation taking place in $\delta_0 y_0 =0$ for $s=0$, whereas $\hatt y_s \notin L^2((0,\pi);dx)$ for $s \in [1,\infty)$ (in accordance with $\tau_s$ being in the limit point case  for $s \in [1,\infty)$) and 
$\hatt y_s^{\, \prime} \notin L^2((0,\pi);dx)$ for $s \in [0,\infty)$. 
\hfill$\diamond$
\end{remark}

We emphasize once again that Theorem~3.1 was originally obtained by Avkhadiev in 2015, see \cite[Lemma~1]{Av15}. 

A closer inspection of the proof of Theorem \ref{t3.1} reveals that $[\sin(x)]^{-2}$ is just a very convenient choice for a function that has inverse square singularities at the interval endpoints as it leads to explicit optimal constants $1/4$ in \eqref{3.4}. To illustrate this point, we consider the differential expressions  
\begin{equation}
\omega_0 = - \f{d^2}{dx^2} - \f{1}{4 x^2}, \quad \alpha_0 = \f{d}{dx} - \f{1}{2x}, \quad 
 \alpha_0^+ = - \f{d}{dx} - \f{1}{2x},\quad x \in (0,\pi),      \lb{3.23}
\end{equation}
such that 
\begin{equation}
\alpha_0^+ \alpha_0 = \omega_0. 
\end{equation}
The minimal and maximal $L^2((0,\pi); dx)$-realizations associated with $\omega_0$ are then given by 
\begin{align}
&S_{0,min} f = \omega_0 f, \no
\\
& f \in \dom(S_{0,min})=\big\{g\in L^2((0,\pi); dx)  \, \big| \,  g,g'\in AC_{loc}((0,\pi));    \\ 
&\hspace*{1.7cm} \supp \, (g)\subset(0,\pi) \text{ is compact; } 
\omega_0 g\in L^2((0,\pi); dx)\big\}.   \no \\
&S_{0,max} f = \omega_0 f,     \no \\
& f \in \dom(S_{0,max})=\big\{g\in L^2((0,\pi); dx) \, \big| \,  g,g'\in AC_{loc}((0,\pi)); \\
& \hspace*{6.3cm}  \omega_0 g\in L^2((0,\pi); dx)\big\},     \no 
\end{align}
implying $S_{0,min}^* = S_{0,max}$, $S_{0,max}^* = S_{0,min}$, 
and we also introduce the following self-adjoint extensions of $S_{0,min}$, respectively, restrictions of $S_{0,max}$ (see, e.g., \cite{AA12}, \cite{AB15}, \cite{AB16}, \cite{BDG11}, \cite{DG21}, \cite{EK07}, \cite{GLN20}, \cite{GP79}, \cite{Ka78}, \cite{KT11}, \cite{Ro85}), 
\begin{align} 
&S_{0,D,N} f = \omega_0 f,     \no \\
& f \in \dom(S_{0,D,N}) = \{g\in \dom(S_{0,max}) \, | \,  \wti g(0)= g' (\pi) = 0 \}     \lb{3.26} \\
& \hspace*{2.6cm} = \big\{g\in \dom(S_{0,max}) \, \big| \, g' (\pi) = 0; \, \alpha_0 g \in L^2((0,\pi); dx)\big\}, 
\no \\
&S_{0,F} f = \omega_0 f,     \no \\
& f \in \dom(S_{0,F}) = \{g\in \dom(S_{0,max}) \, | \,  \wti g(0)= g (\pi) = 0 \}    \lb{3.27} \\
& \hspace*{2.25cm} = \big\{g\in \dom(S_{0,max}) \, \big| \, g(\pi) = 0; \, \alpha_0 g \in L^2((0,\pi); dx)\big\},
\no 
\end{align} 
with $S_{0,F}$ the Friedrichs extension of $S_{0,min}$. The quadratic forms  corresponding to $S_{0,D,N}$ and $S_{0,F}$ are of the form 
\begin{align}
& Q_{S_{0,D,N}} (f,g) = (\alpha_0 f, \alpha_0 g)_{L^2((0,\pi);dx)},   \no \\
& f, g \in \dom(Q_{S_{0,D,N}}) = \big\{g \in L^2((0,\pi); dx) \, \big| \, g \in AC_{loc}((0,\pi)); \, g'(\pi)=0, \\ 
& \hspace*{7.9cm} \alpha_0 g \in L^2((0,\pi); dx)\big\},    \no \\
& Q_{S_{0,F}} (f,g) = (\alpha_0 f, \alpha_0 g)_{L^2((0,\pi);dx)},   \no \\
& f, g \in \dom(Q_{S_{0,F}}) = \big\{g \in L^2((0,\pi); dx) \, \big| \, g \in AC_{loc}((0,\pi)); \, g(\pi) = 0, \\
& \hspace*{7.5cm} \alpha_0 g \in L^2((0,\pi); dx)\big\}.  \no 
\end{align} 
One verifies (see \eqref{B.12}) that for all $\varepsilon > 0$ and $g \in AC_{loc}((0,\varepsilon))$, 
\begin{equation}
\alpha_0 g \in L^2((0,\varepsilon); dx) \, \text{ implies } \, \wti g(0) = 0.
\end{equation}

By inspection,
\begin{align}
\begin{split} 
f_0(\lambda,x) &= x^{1/2} J_0\big(\lambda^{1/2}x\big), \quad x \in (0,\pi), \\
& \hspace*{-0.8mm} \underset{x \downarrow 0}{=} 
x^{1/2} \big[1 + \Oh\big(x^2\big)\big]  
\end{split}
\end{align}
(where $J_{\nu}(\dott)$ denotes the standard Bessel function of order $\nu \in \bbC$, cf.\ \cite[Ch.~9]{AS72})), satisfies (cf.\ \eqref{3.26}, \eqref{3.27})
\begin{equation}
\wti f_0(0) = 0.
\end{equation}
Thus, introducing Lamb's constant, now denoted by $\lambda_{D,N,0}^{1/2}$, as the first positive zero of 
\begin{equation} 
(0,\infty) \ni x \mapsto J_0(x) + 2x J_1(x)
\end{equation} 
(see the brief discussion in \cite{AW07}), one infers that $\lambda_{D,N,0}/\pi^2$ is the first positive zero of 
$f_0'(\dott,\pi)$, that is, 
\begin{equation} 
f_0' (\lambda_{D,N,0}/\pi^2,\pi) = 0. 
\end{equation} 
In addition denoting by $\lambda_{F,0}/\pi^2$ the first strictly positive zero of $f_0(\dott,\pi)$ one has
\begin{equation}
f_0 (\lambda_{F,0}/\pi^2,\pi) = 0,  
\end{equation}
and hence $\lambda_{D,N,0}/\pi^2$ and $\lambda_{F,0}/\pi^2$ are the first eigenvalue of the mixed Dirichlet/Neumann operator $S_{0,D,N}$ and the Dirichlet operator (the Friedrichs extension of 
$S_{0,min}$) $S_{0,F}$, respectively\footnote{In particular, $\lambda_{D,N,0}$ and $\lambda_{F,0}$ are the first eigenvalue of the mixed Dirichlet/Neumann and Dirichlet operator on the interval $(0,1)$.}. Equivalently, 
\begin{equation}
\inf(\sigma(S_{0,D,N})) = \lambda_{D,N,0} \pi^{-2}, \quad \inf(\sigma(S_{0,F})) = \lambda_{F,0} \pi^{-2},
\end{equation}
in particular, 
\begin{align}
& S_{0,D,N} \geq \lambda_{D,N,0} \pi^{-2} I_{L^2((0,\pi);dx)}, \quad 
S_{0,F} \geq \lambda_{F,0} \pi^{-2} I_{L^2((0,\pi);dx)},    \\
& Q_{S_{0,D,N}} (f,f) \geq \lambda_{D,N,0} \pi^{-2} \|f\|^2_{L^2((0,\pi);dx)}, \quad f \in \dom(Q_{S_{0,D,N}}), \\
& Q_{S_{0,F}} (f,f) \geq \lambda_{F,0} \pi^{-2} \|f\|^2_{L^2((0,\pi);dx)}, \quad f \in \dom(Q_{S_{0,F}}). 
\end{align}
Numerically, one confirms that
\begin{equation}
\lambda_{D,N,0} = 0.885..., \quad \lambda_{F,0} = 5.783...\, .
\end{equation}

Thus, arguments analogous to the ones in the proof of Theorem \ref{t3.1} yield the following variants of \eqref{3.1}, \eqref{3.2}, 
\begin{align}
\begin{split} 
\int_0^{\pi} dx \, |f'(x)|^2 &\geq \dfrac{1}{4}\int_0^{\pi} dx \, \dfrac{|f(x)|^2}{x^2} 
+ \f{\lambda_{D,N,0}}{\pi^2} \int_0^{\pi} dx \, |f(x)|^2,   \lb{3.39} \\   
& \hspace*{1.9cm} f \in \dom(Q_{S_{0,D,N}}) \cap H^1((0,\pi)), 
\end{split} \\[2mm]  
\begin{split} 
\int_0^{\pi} dx \, |f'(x)|^2 &\geq \dfrac{1}{4}\int_0^{\pi} dx \, \dfrac{|f(x)|^2}{x^2} 
+ \f{\lambda_{F,0}}{\pi^2} \int_0^{\pi} dx \, |f(x)|^2,    \lb{3.40} \\   
& \hspace*{4cm} f \in H_0^1((0,\pi)), 
\end{split} 
\end{align} 
as well as,
\begin{align}
\begin{split} 
\int_0^{\pi} dx \, |f'(x)|^2 \geq \dfrac{1}{4}\int_0^{\pi} dx \, \dfrac{|f(x)|^2}{d_{(0,\pi)}(x)^2} 
+ \f{4 \lambda_{D,N,0}}{\pi^2} \int_0^{\pi} dx \, |f(x)|^2,&    \lb{3.41} \\   
f \in H_0^1((0,\pi)).& 
\end{split} 
\end{align}
All constants in \eqref{3.39}--\eqref{3.41} are optimal and the inequalities are all strict (for $f \not\equiv 0$).

In obtaining \eqref{3.39}--\eqref{3.41} one makes use of the fact that the domain of a semibounded, self-adjoint operator $A$ in the complex, separable Hilbert space $\cH$ is a form core for $A$, equivalently, a core for $|A|^{1/2}$. In addition, we used (cf.\ \cite[Theorem~7.1]{EK07}) that for $f \in \dom(S_{0,D,N}) 
\cup \dom(S_{0,F})$, there exists $K_0(f) \in \bbC$ such that
\begin{equation}
\lim_{x \downarrow 0}  x^{-1/2} f(x) = K_0(f), \quad \lim_{x \downarrow 0} x^{1/2} f'(x) = K_0(f)/2, 
\quad \lim_{x \downarrow 0} f(x) f'(x) = K_0(f)^2/2.    \lb{3.43} 
\end{equation}
Moreover, if in addition $f' \in L^2((0,1);dx)$, $f(0) = 0$, combining \eqref{3.43} with estimate \eqref{B.16} yields $K_0(f) = 0$, and hence 
$\lim_{x \downarrow} f(x) f'(x) = 0$. This permits one to integrate by parts in $Q_{S_{0,D,N}} (f,f)$ and 
$Q_{S_{0,F}} (f,f)$ and in the process verify \eqref{3.39}--\eqref{3.41}. 

We note that inequalities \eqref{3.39} and \eqref{3.41} were first derived by Avkhadiev and Wirths  \cite{AW07} (also recorded in  \cite{AW11} and \cite[Sect.~3.6.3]{BEL15}; see also \cite{Av21}, \cite{Av21a}, \cite{HHL02}, \cite{NM20}) following a different approach applicable to the multi-dimensional case. We have not found inequality \eqref{3.40} in the literature,  but expect it to be known. 

\begin{remark} \lb{r3.3}
The arguments presented thus far might seem to indicate that Hardy-type inequalities are naturally associated with underlying second-order differential operators satisfying boundary conditions of the Dirichlet and/or Neumann type at the interval endpoints. However, this is not quite the case as the following 
result (borrowed, e.g., from \cite[Lemma~5.3.1]{Da95}, \cite[Sect.~1.1]{OK90}) shows: Suppose $b \in (0,\infty)$, 
$f \in AC_{loc}((0,b))$, $f' \in L^2((0,b);dx)$, $f(0) = 0$, then (with $f$ real-valued without loss of generality), 
\begin{align}
& \int_0^b dx \, |f'(x)|^2 = \int_0^b dx \, \big|x^{1/2} \big[x^{-1/2} f(x)\big]' + (2x)^{-1} f(x)\big|^2 \no \\
& \quad = \int_0^b dx \, \Big\{4^{-1} x^{-2} f(x)^2 + x^{-1/2} f(x) \big[x^{-1/2} f(x)\big]' + 
x \big[\big(x^{-1/2} f(x)\big)'\big]^2\Big\}    \no \\
& \quad \geq \int_0^b dx \, \Big\{4^{-1} x^{-2} f(x)^2 + x^{-1/2} f(x) \big[x^{-1/2} f(x)\big]'\Big\}    \no \\ 
& \quad =  \int_0^b dx \, \f{|f(x)|^2}{4 x^2} + 2^{-1} \big[x^{-1/2} f(x)\big]^2 \big|_{x=0}^b   \no \\
& \quad \geq  \int_0^b dx \, \f{|f(x)|^2}{4 x^2} - 2^{-1} \lim_{x \downarrow 0}\big[x^{-1/2} f(x)\big]^2 \no \\ 
& \quad =  \int_0^b dx \, \f{|f(x)|^2}{4 x^2},       \lb{3.46} 
\end{align} 
employing the estimate \eqref{B.16} with $f(0)=0$. In particular, no boundary conditions whatsoever are needed at the right end point $b$. One notes that the hypotheses on $f$ imply that $f \in AC([0,b])$ and hence actually that $f$ behaves like an $H^1_0$-function in a right neighborhood of $x =0$, equivalently, 
$f \wti \chi_{[0,b/2]} \in H^1_0((0,b))$, where 
\begin{equation}
\wti \chi_{[0,r/2]} (x) = \begin{cases} 1, & x \in [0, r/4], \\
0, & x \in [3r/4,r], 
\end{cases} \quad \wti \chi_{[0,r/2]} \in C^{\infty}([0,r]), \; r \in (0,\infty). 
\end{equation}
\hfill $\diamond$
\end{remark}

\begin{remark} \lb{r3.4} 
Employing locality of the operators involved, one can show (cf.\ \cite{GPS21}) that all considerations in the bulk of this paper, extend to the situation where 
\begin{equation}
q(x) = \f{s^2 - (1/4)}{x^2}, \, \text{ respectively, } \, q(x) = \f{s^2 - (1/4)}{\sin^2(x)}, \quad s \in [0,\infty),
\end{equation}
is replaced by a potential $q$ satisfying $q \in L^1_{loc}((0,\pi);dx)$ and for some 
$s_j \in [0,\infty)$, $j=1,2$, and some $0 < \varepsilon$ sufficiently small, 
\begin{align}
&q_{s_1,s_2}(x) = \begin{cases} \big[s_1^2 - (1/4)\big] x^{-2}, &x \in (0, \varepsilon), \\[1mm] 
\big[s_2^2 - (1/4)\big] (x - \pi)^{-2}, & x \in (\pi - \varepsilon, \pi). 
\end{cases} 
\end{align}
As discussed in \cite{Ka78}, this can be replaced by $q \geq q_{s_1,s_2}$ a.e.

In addition, we only presented the tip of an iceberg in this section as these considerations naturally extend to more general Sturm--Liouville operators in $L^2((a,b); dx)$ generated by differential expressions of the type 
\begin{equation} 
- \f{d}{dx} p(x) \f{d}{dx} + q(x), \quad x \in (a,b), 
\end{equation} 
as discussed to some extent in \cite{GU98}. We will return to this and the general three-coefficient Sturm--Liouville operators in $L^2((a,b); rdx)$ generated by 
\begin{equation} 
\f{1}{r(x)} \bigg[- \f{d}{dx} p(x) \f{d}{dx} + q(x)\bigg], \quad x \in (a,b),  
\end{equation} 
elsewhere. \hfill $\diamond$ 
\end{remark}

\appendix

\section{The Weyl--Titchmarsh--Kodaira $m$-Function \\ Associated with $T_{s,F}$} \lb{sA} 

We start by introducing a normalized fundamental system of solutions $\phi_{0,s}(z,\dott)$ and 
$\theta_{0,s}(z,\dott)$ of $\tau_s u=zu$, $s \in [0,1)$, $z \in \bbC$, satisfying (cf.\ the generalized boundary values introduced in \eqref{2.13}, \eqref{2.14})
\begin{align}
\wti \theta_{0,s}(z,0)=1,\quad \wti \theta^{\, \prime}_{0,s}(z,0)=0,  \quad 
\wti \phi_{0,s}(z,0)=0,\quad \wti \phi^{\, \prime}_{0,s}(z,0)=1, 
\end{align}
with $\phi_{0,s}(\dott,x)$ and $\theta_{0,s}(\dott,x)$ entire for fixed $x \in (0,\pi)$. 
To this end, we introduce the two linearly independent solutions to $\tau_s y=zy$ (entire w.r.t. $z$ for fixed 
$x \in (0,\pi)$) given by
\begin{align}  \lb{A.2}
y_{1,s}(z,x)&=[\sin (x)]^{(1-2s)/2}    \no \\ 
& \quad \times F\big(\big[(1/2) - s  + z^{1/2}\big]\big/2,\big[(1/2) - s  - z^{1/2}\big]\big/2;1/2;\cos^2 (x)\big), \no \\
y_{2,s}(z,x)&=\cos (x) [\sin (x)]^{(1-2s)/2}      \\
& \quad \times F\big(\big[(3/2) - s  + z^{1/2}\big]\big/2,\big[(3/2) - s  - z^{1/2}\big]\big/2;3/2;\cos^2 (x)\big), \no \\
&\hspace{5.55cm} s\in[0,1),\; z\in\bbC,\; x\in(0,\pi).    \no 
\end{align}
Using the connection formula found in \cite[Eq. 15.3.6]{AS72} yields the behavior near $x=0,\pi,$
\begin{align}
y_{1,s}(z,x)&=[\sin (x)]^{(1-2s)/2}\dfrac{\pi^{1/2}\Gamma(s)}{\Gamma\big(\big[(1/2) + s  + z^{1/2}\big]\big/2\big)\Gamma\big(\big[(1/2) + s  - z^{1/2}\big]\big/2\big)} \no \\
&\qquad \times F\big(\big[(1/2) - s  + z^{1/2}\big]\big/2,\big[(1/2) - s  - z^{1/2}\big]\big/2;1-s;\sin^2 (x)\big) \no \\
&\quad\; + [\sin (x)]^{(1+2s)/2}\dfrac{\pi^{1/2}\Gamma(-s)}{\Gamma\big(\big[(1/2) - s  + z^{1/2}\big]\big/2\big)\Gamma\big(\big[(1/2) - s  - z^{1/2}\big]\big/2\big)} \no \\
&\qquad \times F\big(\big[(1/2) + s  + z^{1/2}\big]\big/2,\big[(1/2) + s  - z^{1/2}\big]\big/2;1+s;\sin^2 (x)\big), \no \\ 
&\hspace{5.6cm} s\in(0,1),\; z\in\bbC,\; x\in(0,\pi),  \lb{A.3} \\[1mm]
y_{2,s}(z,x)&=\cos (x) [\sin (x)]^{(1-2s)/2}    \no \\
& \qquad \times \dfrac{\pi^{1/2}\Gamma(s)}{2\Gamma\big(\big[(3/2) 
+ s  + z^{1/2}\big]\big/2\big)\Gamma\big(\big[(3/2) + s  - z^{1/2}\big]\big/2\big)} \no \\
&\qquad \times F\big(\big[(3/2) - s  + z^{1/2}\big]\big/2,\big[(3/2) - s  - z^{1/2}\big]\big/2;1-s;\sin^2 (x)\big) \no \\
&\quad\; + \cos (x) [\sin (x)]^{(1+2s)/2}      \no \\ 
& \qquad \times \dfrac{\pi^{1/2}\Gamma(-s)}{2\Gamma\big(\big[(3/2) - s  + z^{1/2}\big]\big/2\big)\Gamma\big(\big[(3/2) - s  - z^{1/2}\big]\big/2\big)} \no \\
&\qquad \times F\big(\big[(3/2) + s  + z^{1/2}\big]\big/2,\big[(3/2) + s  - z^{1/2}\big]\big/2;1+s;\sin^2 (x)\big), \no \\[1mm]
&\hspace{5.cm} s\in(0,1),\; z\in\bbC,\; x\in(0,\pi)\backslash\{\pi/2\}.   \no 
\end{align}

\begin{remark} \lb{rA.1}
Before we turn to the case $s=0$, we recall Gauss's identity (cf.\ \cite[no.~15.1.20]{AS72})
\begin{equation}
F(\alpha,\beta;\gamma;1) = \f{\Gamma(\gamma) \Gamma(\gamma-\alpha-\beta)}{\Gamma(\gamma-\alpha) 
\Gamma(\gamma-\beta)}, \quad \gamma \in\bbC \backslash \{ - \bbN_0\}, 
\; \Re(\gamma-\alpha-\beta) > 0,    \lb{A.3a} 
\end{equation} 
and the differentiation formula (cf.\ \cite[no.~15.2.1]{AS72})
\begin{equation}
\f{d}{dz} F(\alpha,\beta;\gamma;z) = \f{\alpha\beta}{\gamma} F(\alpha+1,\beta+1;\gamma+1;z), 
\quad \alpha, \beta, \gamma \in \bbC, \; z \in \{\zeta \in \bbC \,|\, |\zeta| < 1\}, 
\lb{A.3b} 
\end{equation}
which imply that for $s \in (0,1)$, the two $F(\dott,\dott;\dott;1)$ exist in \eqref{A.2} (indeed, for $j=1, 2$ one obtains  with $s \in (0,1)$, and with the values of $\alpha,\beta,\gamma$ taken from \eqref{A.2}, that 
$\gamma-\alpha-\beta = s > 0$) and hence the asymptotic behavior of $y_{j,s}(z,x)$, $j=1,2$, as $x \downarrow 0$ and $x \uparrow \pi$ is dominated by $x^{(1-2s)/2}$ and 
$(\pi-x)^{(1-2s)/2}$, respectively. However, the analogous statement fails for $y_{j,s}'(z,x)$, $j=1,2$, as, taking into account \eqref{A.3b}, the analog of the 2nd condition in \eqref{A.3a}, namely, $\Re[\gamma+1-(\alpha+1)-(\beta+1)]>0$, is not fulfilled (in this case the values of $\alpha,\beta,\gamma$ taken from \eqref{A.2} with $s \in (0,1)$ yield 
$[\gamma+1 - (\alpha+1) - (\beta+1)] = s - 1 < 0$). The situation is similar for the first two $F(\dott,\dott;\dott;x)$ for $y_{1,s}(z,x)$ in \eqref{A.3} as $x \to \pi/2$ as in this case the two $F(\dott,\dott;\dott;1)$ exist. Even though for $y_{2,s}(z,x)$ in \eqref{A.3} the two $F(\dott,\dott;\dott;1)$ do not exist individually, the limit of each term does exist due to the multiplication by the factor $\cos(x)$. To see this, one can instead consider the limit (cf. \cite[no.~15.4.23]{OLBC10})
\begin{align}
\lim_{z\to1^-}\dfrac{F(\alpha,\beta;\gamma;z)}{(1-z)^{\gamma-\alpha-\beta}}=\dfrac{\Gamma(\gamma) 
\Gamma(\alpha+\beta-\gamma)}{\Gamma(\alpha)\Gamma(\beta)},\quad \gamma \in\bbC \backslash \{ - \bbN_0\}, 
\; \Re(\gamma-\alpha-\beta) < 0,
\end{align}
which through the appropriate change of variable reveals that the connection formula for $y_{2,s}(z,x)$ as $x \to \pi/2$ approaches $0$ as expected from evaluating $y_{2,s}(z,\pi/2)$ in \eqref{A.2}. But once again, the analog of the 2nd condition in \eqref{A.3a}, namely, $\Re[\gamma+1-(\alpha+1)-(\beta+1)] > 0$, fails for the four $F'(\dott,\dott;\dott;x)$ in \eqref{A.3} as $x \to \pi/2$. 
${}$ \hfill $\diamond$
\end{remark}

Similarly, by \cite[Eq. 15.3.10]{AS72} one obtains for the remaining case $s=0$, 
\begin{align}
&y_{1,0}(z,x)=\dfrac{\pi^{1/2}[\sin (x)]^{1/2}}{\Gamma\big(\big[(1/2) + z^{1/2}\big]\big/2\big)\Gamma\big(\big[(1/2) - z^{1/2}\big]\big/2\big)}
\no \\
& \hspace*{1.7cm}  \times 
\sum_{n=0}^\infty \dfrac{\big(\big[(1/2) + z^{1/2}\big]\big/2\big)_n\big(\big[(1/2) - z^{1/2}\big]\big/2\big)_n}{(n!)^2} 
 \big[2\psi(n+1)    \no \\
&\hspace*{2cm}  -\psi\big(n+\big[(1/2) + z^{1/2}\big]\big/2\big)-\psi\big(n+\big[(1/2) - z^{1/2}\big]\big/2\big)    \no \\
& \hspace*{2cm} -\ln(\sin^2 (x))\big]
[\sin (x)]^{2n},     \lb{A.4}  \\[1mm]
&y_{2,0}(z,x)=\dfrac{\pi^{1/2}\cos (x) [\sin (x)]^{1/2}}{2\Gamma\big(\big[(3/2) + z^{1/2}\big]\big/2\big)\Gamma\big(\big[(3/2) - z^{1/2}\big]\big/2\big)}    \no \\
& \hspace*{1.7cm}  \times 
\sum_{n=0}^\infty \dfrac{\big(\big[(3/2) + z^{1/2}\big]\big/2\big)_n\big(\big[(3/2) - z^{1/2}\big]\big/2\big)_n}{(n!)^2}  \big[2\psi(n+1)
\no \\
& \hspace*{2cm}  -\psi\big(n+\big[(3/2) + z^{1/2}\big]\big/2\big)-\psi\big(n+\big[(3/2) - z^{1/2}\big]\big/2\big)   \no \\
& \hspace*{2cm} - \ln\big(\sin^2 (x)\big)\big][\sin (x)]^{2n}, \no \\[1mm]
& \hspace*{2.15cm} s=0,\; z\in\bbC,\; x\in(0,\pi). \no
\end{align}
Here $\psi(\dott) = \Gamma'(\dott)/\Gamma(\dott)$ denotes the Digamma function, $\gamma_{E} = - \psi(1) = 0.57721\dots$ represents Euler's constant, and 
\begin{equation}
(\zeta)_0 =1, \quad (\zeta)_n = \Gamma(\zeta + n)/\Gamma(\zeta), \; n \in \bbN, 
\quad \zeta \in \bbC \backslash (-\bbN_0),      \lb{A.5} 
\end{equation}
abbreviates Pochhammer's symbol $($see, e.g., \cite[Ch.~6]{AS72}$)$. 
Direct computation now yields
\begin{align}
\wti y_{1,s}(z,0)&=-\wti y_{1,s}(z,\pi)=\dfrac{2\pi^{1/2}\Gamma(1+s)}{\Gamma\big(\big[(1/2) + s  + z^{1/2}\big]\big/2\big)\Gamma\big(\big[(1/2) + s  - z^{1/2}\big]\big/2\big)},      \no \\[1mm]
\wti y^{\, \prime}_{1,s}(z,0)&=\wti y^{\, \prime}_{1,s}(z,\pi)=\dfrac{\pi^{1/2}\Gamma(-s)}{\Gamma\big(\big[(1/2) - s  + z^{1/2}\big]\big/2\big)\Gamma\big(\big[(1/2) - s  - z^{1/2}\big]\big/2\big)},       \no \\[1mm]
\wti y_{2,s}(z,0)&=\wti y_{2,s}(z,\pi)=\dfrac{\pi^{1/2}\Gamma(1+s)}{\Gamma\big(\big[(3/2) + s  + z^{1/2}\big]\big/2\big)\Gamma\big(\big[(3/2) + s  - z^{1/2}\big]\big/2\big)},      \no \\[1mm]
\wti y^{\, \prime}_{2,s}(z,0)&=-\wti y^{\, \prime}_{2,s}(z,\pi)
=\dfrac{\pi^{1/2}\Gamma(-s)}{2\Gamma\big(\big[(3/2) - s  + z^{1/2}\big]\big/2\big)\Gamma\big(\big[(3/2) - s  - z^{1/2}\big]\big/2\big)},      \no \\[1mm]
&\hspace*{7cm} s\in(0,1),\, z\in\bbC,    \lb{A.6} \\
\begin{split}\lb{A.7}
\wti y_{1,0}(z,0)&=-\wti y_{1,0}(z,\pi)=\dfrac{2\pi^{1/2}}{\Gamma\big(\big[(1/2) + z^{1/2}\big]\big/2\big)\Gamma\big(\big[(1/2) - z^{1/2}\big]\big/2\big)},\\[1mm]
\wti y^{\, \prime}_{1,0}(z,0)&=\wti y^{\, \prime}_{1,0}(z,\pi)     \\
& =\dfrac{-\pi^{1/2}\big[2\gamma_E+\psi\big(\big[(1/2) + z^{1/2}\big]\big/2\big)+\psi\big(\big[(1/2) - z^{1/2}\big]\big/2\big)\big]}{\Gamma\big(\big[(1/2) + z^{1/2}\big]\big/2\big)\Gamma\big(\big[(1/2) - z^{1/2}\big]\big/2\big)},\\[1mm]
\wti y_{2,0}(z,0)&=\wti y_{2,0}(z,\pi)=\dfrac{\pi^{1/2}}{\Gamma\big(\big[(3/2) + z^{1/2}\big]\big/2\big)\Gamma\big(\big[(3/2) - z^{1/2}\big]\big/2\big)},\\[1mm]
\wti y^{\, \prime}_{2,0}(z,0)&=-\wti y^{\, \prime}_{2,0}(z,\pi)    \\ 
& =\dfrac{-\pi^{1/2}\big[2\gamma_E+\psi\big(\big[(3/2) + z^{1/2}\big]\big/2\big)+\psi\big(\big[(3/2) - z^{1/2}\big]\big/2\big)\big]}{2\Gamma\big(\big[(3/2) + z^{1/2}\big]\big/2\big)\Gamma\big(\big[(3/2) - z^{1/2}\big]\big/2\big)},\\[1mm]
&\hspace*{7.5cm} s=0,\, z\in\bbC.
\end{split}
\end{align}
In particular, one obtains
\begin{align}
\begin{split}
\phi_{0,s}(z,x)&= \wti y_{2,s}(z,0)y_{1,s}(z,x) - \wti y_{1,s}(z,0) y_{2,s}(z,x),     \\[1mm]
\theta_{0,s}(z,x)&= \wti y^{\, \prime}_{1,s}(z,0) y_{2,s}(z,x) - \wti y^{\, \prime}_{2,s}(z,0)y_{1,s}(z,x),  \\[1mm]
&\hspace{2.1cm} s\in[0,1),\; z\in\bbC,\; x\in(0,\pi),
\end{split}
\end{align}
since 
\begin{align}
W(y_{1,s}(z,\dott), y_{2,s}(z,\dott)) = \wti y_{1,s}(z,0)\wti y^{\, \prime}_{2,s}(z,0) 
- \wti y^{\, \prime}_{1,s}(z,0)\wti y_{2,s}(z,0) = - 1,     \lb{A.11}
\end{align}
with the generalized boundary values given by \eqref{A.6}, \eqref{A.7}. To prove \eqref{A.11} one recalls Euler's reflection formula (cf. \cite[no.~6.1.17]{AS72})
\begin{align}
\Gamma(z)\Gamma(1-z)=\dfrac{\pi}{\sin(\pi z)}, \quad z \in \bbC \backslash \bbZ, 
\end{align}
and hence concludes that
\begin{align}
\begin{split} 
& \Gamma\big(\big[(1/2) + \varepsilon s \pm z^{1/2}\big]\big/2\big)\Gamma\big(\big[(3/2) - \varepsilon s \mp z^{1/2}\big]\big/2\big)   \\
& \quad =\dfrac{\pi}{\sin\big(\pi\big[(1/2) + \varepsilon \pm z^{1/2}\big]\big/2\big)}, \quad \varepsilon \in \{-1,1\}.
\end{split} 
\end{align}
Thus one computes for $s\in(0,1)$,
\begin{align}
& W(y_{1,s}(z,\dott), y_{2,s}(z,\dott))     \no \\
& \quad =-[\sin(\pi s)]^{-1}\big\{\sin\big(\pi\big[(1/2)+s + z^{1/2}\big]\big/2\big)\sin\big(\pi\big[(1/2)+s - z^{1/2}\big]\big/2\big)  \no \\
& \quad \hspace*{2.5cm}-\sin\big(\pi\big[(1/2)-s + z^{1/2}\big]\big/2\big)\sin\big(\pi\big[(1/2)-s - z^{1/2}\big]\big/2\big) \big\}  \no\\
& \quad =-[2\sin(\pi s)]^{-1}\{-\cos(\pi[(1/2)+s ])+\cos(\pi[(1/2)-s ]) \} = -1.     
\end{align} 
For the case $s=0$, one recalls the reflection formula for the Digamma function (cf. \cite[no.~6.3.7]{AS72})
\begin{align}
\psi(1-z)-\psi(z)=\pi\cot(\pi z), \quad z \in \bbC \backslash \bbZ, 
\end{align}
and applies trigonometric identities to obtain $W(y_{1,0}(z,\dott), y_{2,0}(z,\dott)) = -1$.

The singular Weyl--Titchmarsh--Kodaira function $m_{0,0,s}(z)$ is then uniquely determined (cf. \cite[Eq. (3.18)]{GLPS21} and \cite{GLN20} for background on $m$-functions) to be\footnote{Here the subscripts $0,0$ in $m_{0,0,s}$ indicate the Dirichlet (i.e., Friedrichs) boundary conditions at $x=0, \pi$, a special case of the $m_{\alpha,\beta}$-function discussed in \cite{GLN20} associated with separated boundary conditions at $x=0, \pi$, indexed by boundary condition parameters $\alpha, \beta \in [0,\pi]$.}
\begin{align}
m_{0,0,s}(z)= 
- \f{\wti\theta_{0,s} (z,\pi)}{\wti\phi_{0,s} (z,\pi)}, \quad s\in[0,1),\; z\in\rho(T_{s,F}).   
\end{align}
Direct calculation once again yields
\begin{align}
m_{0,0}(z)&= -\dfrac{\wti y^{\, \prime}_{2,s}(z,0)\wti y_{1,s}(z,\pi)-\wti y^{\, \prime}_{1,s}(z,0)\wti y_{2,s}(z,\pi)}{2\wti y_{1,s}(z,0)\wti y_{2,s}(z,0)} \no \\[1mm]
&=\begin{cases}
\dfrac{\pi\Gamma(-s)}{4\Gamma(1+s)}\bigg[\dfrac{\Gamma\big(\big[(3/2) + s  + z^{1/2}\big]\big/2\big)\Gamma\big(\big[(3/2) + s  - z^{1/2}\big]\big/2\big)}{\Gamma\big(\big[(3/2) - s  + z^{1/2}\big]\big/2\big)\Gamma\big(\big[(3/2) - s  - z^{1/2}\big]\big/2\big)} \\[3mm]
\qquad\qquad\quad +\dfrac{\Gamma\big(\big[(1/2) + s  + z^{1/2}\big]\big/2\big)\Gamma\big(\big[(1/2) + s  - z^{1/2}\big]\big/2\big)}{\Gamma\big(\big[(1/2) - s  + z^{1/2}\big]\big/2\big)\Gamma\big(\big[(1/2) - s  - z^{1/2}\big]\big/2\big)}\bigg], \\
\hspace*{8.2cm} s\in(0,1),\\[3mm]
-\big[4\gamma_E+\psi\big(\big[(1/2) + z^{1/2}\big]\big/2\big)+\psi\big(\big[(1/2) - z^{1/2}\big]\big/2\big) \\[1mm]
\quad \, +\psi\big(\big[(3/2) + z^{1/2}\big]\big/2\big)+\psi\big(\big[(3/2) - z^{1/2}\big]\big/2\big)\big]/4, \quad s=0,
\end{cases} \no \\
&\hspace*{8.15cm} z\in\rho(T_{s,F}),
\end{align}
which has simple poles precisely at the simple eigenvalues of $T_{s,F}$ given by
\begin{align}\lb{A.12}
\sigma(T_{s,F})=\big\{[(1/2)+s+n]^2\big\}_{n\in\bbN_0},\quad s\in[0,1).
\end{align}

\begin{remark} \lb{rA.2} 
For the limit point case at both endpoints, that is, for $s \in [1,\infty)$, the solutions $y_{j,s}(z,\dott)$ in \eqref{A.2} remain linearly independent and also the connection formulas \eqref{A.3} remain valid for 
$s \in [1,\infty) \backslash \bbN$. Moreover, employing once again \eqref{A.3a} and \eqref{A.3b} one verifies 
that the two $F(\dott,\dott;\dott;1)$ as well as $F'(\dott,\dott;\dott;1)$ are well defined in \eqref{A.2} and hence for 
$s \in [1,\infty)$, the asymptotic behavior of $y_{j,s}(z,x)$ and $y_{j,s}'(z,x)$, $j=1,2$, as $x \downarrow 0$ and $x \uparrow \pi$ is dominated by $x^{(1-2s)/2}$ and $x^{-(1+2s)/2}$ and $(\pi-x)^{(1-2s)/2}$ and 
$(\pi-x)^{-(1+2s)/2}$, respectively. Since in connection with \eqref{A.3} one then has $\gamma-\alpha-\beta=\pm1/2$, independently of the value of $s \in (0,\infty)$, the situation described in Remark \ref{rA.1} for \eqref{A.3} and 
$s \in (0,1)$ applies without change to the current case $s \in [1,\infty)$. 

Actually, some of these failures (as $x \to \pi/2$ in $y_{j,s}'(z,x)$, $j=1,2$) are crucial for the following elementary reason: The function 
\begin{align}
& [\sin (x)]^{(1+2s)/2}\dfrac{\pi^{1/2}\Gamma(-s)}{\Gamma\big(\big[(1/2) - s  + z^{1/2}\big]\big/2\big)\Gamma\big(\big[(1/2) - s  - z^{1/2}\big]\big/2\big)} \no \\
&\quad \times F\big(\big[(1/2) + s  + z^{1/2}\big]\big/2,\big[(1/2) + s  - z^{1/2}\big]\big/2;1+s;\sin^2 (x)\big),  \\
&\hspace{5.6cm} s\in [1,\infty),\; z\in\bbC,\; x\in(0,\pi),    \no 
\end{align}
(i.e., the analog of the second part of $y_{1,s}(z,\dott)$ on the right-hand side in \eqref{A.3}) generates an $L^2((0,\pi); dx)$-element near $x=0,\pi$, and hence if this function and its $x$-derivative were locally absolutely continuous in a neighborhood of $x = \pi/2$ (the only possibly nontrivial point in the interval $(0,\pi)$), the self-adjoint maximal operator $T_{s,max}$, $s \in [1,\infty)$, would have eigenvalues for all $z \in \bbC$, an obvious contradiction.
\hfill $\diamond$
\end{remark}

Because of the subtlety pointed out in Remark \ref{rA.2} we omit further details on the limit point case 
$s \in [1,\infty)$ and refer to \cite[Sect. 4]{GK85}, instead. In particular, \cite[Theorem~4.1\,b)]{GK85} extends \eqref{A.12} to $s \in [1,\infty)$ and hence one actually has
\begin{align}\lb{A.13}
\sigma(T_{s,F})=\big\{[(1/2)+s+n]^2\big\}_{n\in\bbN_0}, 
\quad s\in[0,\infty).
\end{align}

\section{Remarks on Hardy-Type Inequalities} \lb{sB} 

In this appendix we recall a Hardy-type inequality useful in Section \ref{s2}.

Introducing the differential expressions $\alpha_s$, 
$\alpha^+_s$ (cf.\ \eqref{3.23} for $s = 0$),
\begin{equation}
\alpha_s = \f{d}{dx} - \f{s + (1/2)}{x}, \quad \alpha_s^+ = - \f{d}{dx} - \f{s + (1/2)}{x}, \quad 
s \in [0,\infty), \; x \in (0,\pi),      \lb{B.7}
\end{equation}
one confirms that 
\begin{equation}
\alpha_s^+ \alpha_s = \omega_s = - \f{d^2}{dx^2} + \f{s^2 - (1/4)}{x^2}, \quad s \in [0,\infty), 
\; x \in (0,\pi).     \lb{B.8} 
\end{equation}

Following the Hardy inequality considerations in \cite{GP79}, \cite{Ka72}, \cite{KW72}, one obtains the following 
basic facts.

\begin{lemma} \lb{lB.1} 
Suppose $f \in AC_{loc}((0,\pi))$, $\alpha_s f \in L^2((0,\pi); dx)$ for some $s \in\bbR$, and 
$0 < r_0 < r_1 < \pi < R < \infty$. Then, 
\begin{align}
\begin{split} 
& \int_{r_0}^{r_1} dx \, |(\alpha_s f)(x)|^2 \geq s^2 \int_{r_0}^{r_1} dx \, \f{|f(x)|^2}{x^2} + 
\f{1}{4} \int_{r_0}^{r_1} dx \, \f{|f(x)|^2}{x^2 [\ln(R/x)]^2}    \lb{B.9} \\
& \hspace*{3.25cm} - s \f{|f(x)|^2}{x}\bigg|_{x = r_0}^{r_1} - \f{|f(x)|^2}{2 x [\ln(R/x)]}\bigg|_{x = r_0}^{r_1},  
\end{split} \\
& \int_{r_0}^{r_1} dx \, x \ln(R/x) \bigg|\bigg[\f{f(x)}{x^{1/2} [\ln(R/x)]^{1/2}}\bigg]'\bigg|^2    \no \\
& \quad = \int_{r_0}^{r_1} dx \, \bigg[|f'(x)|^2 - \f{|f(x)|^2}{4 x^2} 
- \f{|f(x)|^2}{4 x^2 [\ln(R/x)]^2}\bigg]     \lb{B.10} \\
& \qquad - \f{|f(x)|^2}{2x}\bigg|_{x = r_0}^{r_1} + \f{|f(x)|^2}{2x \ln(R/x)}\bigg|_{x = r_0}^{r_1} \geq 0,   \no \\
\begin{split} 
& \int_{r_0}^{r_1} dx \, |(\alpha_s f)(x)|^2 =  \int_{r_0}^{r_1} dx \, \bigg[|f'(x)|^2 
+ \big[s^2 - (1/4)\big] \f{|f(x)|^2}{x^2}\bigg]     \lb{B.11} \\
& \hspace*{3.25cm} - [s + (1/2)] \f{|f(x)|^2}{x}\bigg|_{x = r_0}^{r_1} \geq 0.    
\end{split} 
\end{align}
If $s = 0$, 
\begin{equation}
\int_0^{r_1} dx \, \f{|f(x)|^2}{x^2 [\ln(R/x)]^2} < \infty, \quad  
\lim_{x \downarrow 0} \f{|f(x)|}{[x \ln(R/x)]^{1/2}} = 0.    \lb{B.12} 
\end{equation}
If $s \in (0,\infty)$, then  
\begin{equation}
\int_0^{r_1} dx \, |f'(x)|^2 < \infty, \quad \int_0^{r_1} dx \, \f{|f(x)|^2}{x^2} < \infty, \quad 
\lim_{x \downarrow 0} \f{|f(x)|}{x^{1/2}} = 0,     \lb{B.13} 
\end{equation}
in particular, 
\begin{equation}
f \wti \chi_{[0,r_1/2]} \in H^1_0((0,r_1)),
\end{equation}
where
\begin{equation}
\wti \chi_{[0,r/2]} (x) = \begin{cases} 1, & x \in [0,r/4], \\
0, & x \in [3r/4,r], 
\end{cases} \quad \wti \chi_{[0,r/2]} \in C^{\infty}([0,r]), \; r \in (0,\infty). 
\end{equation}
\end{lemma}
\begin{proof}
Relations \eqref{B.10} and \eqref{B.11} are straightforward (yet somewhat tedious) identities; together 
they yield \eqref{B.9}. The 1st relation in \eqref{B.12} is an instant consequence of \eqref{B.9}, so is the fact that $\lim_{x \downarrow 0} |f(x)|^2 / [x \ln(R/x)]$ exists. Moreover, since $[x \ln(R/x)]^{-1}$ is not integrable at $x = 0$, the 1st relation in \eqref{B.12} yields 
$\liminf_{x \downarrow 0} |f(x)|^2 / [x \ln(R/x)] = 0$, implying the 2nd relation in \eqref{B.12}.  

Finally, if $s \in (0,\infty)$, then inequality \eqref{B.9} implies the 2nd relation in \eqref{B.13}; together with 
$\alpha_s f \in  L^2((0,\pi); dx)$, this yields the 1st relation in \eqref{B.13}. By inequality \eqref{B.9}, 
$\lim_{x \downarrow 0} |f(x)|^2 /x$ exists, but then the second relation in \eqref{B.13} yields 
$\liminf_{x \downarrow 0} |f(x)|^2 /x = 0$ and hence also $\lim_{x \downarrow 0} |f(x)|^2 /x = 0$. 
\end{proof}

We also recall the following elementary fact.

\begin{lemma} \lb{lB.2}
Suppose $f \in H^1((0,r))$ for some $r \in (0,\infty)$. Then, for all $x \in (0,r)$, 
\begin{align}
\begin{split} 
|f(x) - f(0)| &= \bigg|\int_0^x dt \, f'(t)\bigg| \leq x^{1/2} \bigg(\int_0^x dt \, |f'(t)|^2\bigg)^{1/2}    \\
& \leq x^{1/2} \|f'\|_{L^2((0,x);dt)} \underset{x \downarrow 0}{=} \oh\big(x^{1/2}\big).     \lb{B.16}
\end{split} 
\end{align} 
Thus, if $f \in H^1((0,r))$, then $\int_0^r dx \, |f(x)|^2/x^2 < \infty$ if and only if $f(0) = 0$, that is, if and only if $f \wti \chi_{[0,r/2]} \in H^1_0((0,r))$. 

In particular, if $f \in H^1((0,r))$ and $f(0) = 0$, then actually,  
\begin{equation}
\lim_{x \downarrow 0} \f{|f(x)|}{x^{1/2}} = 0.     \lb{B.17} 
\end{equation}
\end{lemma} 
\begin{proof}
Since \eqref{B.16} is obvious, we briefly discuss the remaining assertions in Lemma \ref{lB.2}. If 
$f \in H^1((0,r))$ and $\int_0^r dx \, |f(x)|^2/x^2 < \infty$ then identity \eqref{B.11} for $s < - 1/2$, that is, 
\begin{align}
\begin{split} 
& \int_{r_0}^{r_1} dx \, |(\alpha_s f)(x)|^2 =  \int_{r_0}^{r_1} dx \, \bigg[|f'(x)|^2 
+ \big[s^2 - (1/4)\big] \f{|f(x)|^2}{x^2}\bigg]     \lb{B.18} \\
& \hspace*{3.25cm} - [s + (1/2)] \f{|f(x)|^2}{x}\bigg|_{x = r_0}^{r_1} \geq 0, \quad s < - 1/2,      
\end{split} 
\end{align} 
yields the existence of $\lim_{x \downarrow 0} |f(x)|^2/x$. Since $\int_0^r dx \, |f(x)|^2/x^2 < \infty$ 
implies that  $\liminf_{x \downarrow 0} |f(x)|^2/x = 0$, one concludes that $\lim_{x \downarrow 0} |f(x)|^2/x = 0$ 
and hence $f$ behaves locally like an $H^1_0$-function in a right neighborhood of $x = 0$. Conversely, if 
$f(0) = 0$, then $\int_0^r dx \, |f(x)|^2/x^2 < \infty$ by Hardy's inequality as discussed in 
Remark \ref{r3.3}. Relation \eqref{B.17} is clear from \eqref{B.16} with $f(0)=0$. 
\end{proof}

\begin{remark} \lb{rB.3} 
$(i)$ If $f \in AC_{loc}((0,r))$ and $f' \in L^p((0,r); dx)$ for some $p \in [1,\infty)$, the H\"older estimate analogous to \eqref{B.16},
\begin{align}
\begin{split} 
|f(d) - f(c)| = \bigg|\int_c^d dt \, f'(t) \bigg| \leq |d-c|^{1/p'} \bigg(\int_c^d dt \, |f'(t)|^p\bigg)^{1/p}, \\ 
(c,d) \subset (0,r), \; \f{1}{p} + \f{1}{p'} = 1, 
\end{split} 
\end{align} 
implies the existence of $\lim_{c \downarrow 0} f(c) = f(0)$ and $\lim_{d \uparrow r} f(d) = f(r)$ and hence yields 
$f \in AC([0,r])$. \\[1mm] 
$(ii)$ The fact that $f \in H^1((0,r))$ and $\int_0^r dx \, |f(x)|^2/x^2 < \infty$ implies $f \wti \chi_{[0,r/2]} \in H^1_0((0,r))$ is a special case of a multi-dimensional result recorded, for instance, in 
\cite[Theorem~5.3.4]{EE18}. \\[1mm]
$(iii)$ When replacing $x^{-2}$, $x \in (0,r)$, by $[\sin(x)]^{-2}$, $x \in (0,\pi)$, due to locality, the considerations in Lemmas \ref{lB.1} and \ref{lB.2} at the left endpoint $x=0$ apply of course to the right interval endpoint $\pi$. 
${}$ \hfill $\diamond$ 
\end{remark}
 
\medskip

\noindent {\bf Acknowledgments.}
We are indebted to Jan Derezinski, Aleksey Kostenko, Ari Laptev, and Gerald Teschl for very helpful discussions and to Farit Avkhadiev for kindly pointing out to us references \cite{Av15} and \cite{Av21a}. 


\end{document}